\documentclass[12pt,a4paper,reqno]{amsart}

\title{Twisted Ehresmann Schauenburg bialgebroids}
\usepackage{amsmath, amssymb, amsthm}
\usepackage[english]{babel}
\usepackage{fullpage}
\usepackage{enumerate}
\usepackage[all]{xy}
\usepackage{graphicx}
\usepackage{mathtools}
\usepackage{mathrsfs}
\usepackage{hyperref}
\usepackage{verbatim}

\author{Xiao Han*}
\address[]{\textit{Xiao Han}
\newline \indent SISSA,
via Bonomea 265, 34136 Trieste, Italy}
\email{xihan@sissa.it}

\DeclareMathOperator{\Hom}{Hom}

\newcommand*{\id}{\textup{id}}

\numberwithin{equation}{section}

\theoremstyle{plain}

\newtheorem{thm}{Theorem}[section]
\newtheorem{lem}[thm]{Lemma}
\newtheorem{prop}[thm]{Proposition}
 \newtheorem{cor}[thm]{Corollary}
\newtheorem{defi}[thm]{Definition}

\theoremstyle{remark}

\newtheorem{rem}[thm]{Remark}

\numberwithin{equation}{section}

\newcommand{\ot}{\otimes}

\newcommand{\beq}{\begin{equation}}
\newcommand{\eeq}{\end{equation}}

\newcommand{\cL}{\mathcal{L}}

\newcommand{\C}{\mathcal{C}}

\newcommand{\zero}[1]{{#1}{}_{\scriptscriptstyle{(0)}}}
\newcommand{\one}[1]{{#1}{}_{\scriptscriptstyle{(1)}}}
\newcommand{\two}[1]{{#1}{}_{\scriptscriptstyle{(2)}}}
\newcommand{\three}[1]{{#1}{}_{\scriptscriptstyle{(3)}}}
\newcommand{\four}[1]{{#1}{}_{\scriptscriptstyle{(4)}}}
\newcommand{\five}[1]{{#1}{}_{\scriptscriptstyle{(5)}}}
\newcommand{\six}[1]{{#1}{}_{\scriptscriptstyle{(6)}}}
\newcommand{\seven}[1]{{#1}{}_{\scriptscriptstyle{(7)}}}
\newcommand{\eight}[1]{{#1}{}_{\scriptscriptstyle{(8)}}}
\newcommand{\nine}[1]{{#1}{}_{\scriptscriptstyle{(9)}}}
\newcommand{\ten}[1]{{#1}{}_{\scriptscriptstyle{(10)}}}
\newcommand{\eleven}[1]{{#1}{}_{\scriptscriptstyle{(11)}}}
\newcommand{\twelve}[1]{{#1}{}_{\scriptscriptstyle{(12)}}}

\newcommand{\tuno}[1]{{#1}{}{}^{\scriptscriptstyle{<1>}}}
\newcommand{\tdue}[1]{{#1}{}{}^{\scriptscriptstyle{<2>}}}

\begin{document}

\begin{abstract}
We construct an invertible normalised 2 cocycle on the Ehresmann Schauenburg bialgebroid of a cleft Hopf Galois extension under the condition that the corresponding Hopf algebra is cocommutative and the image of the unital cocycle corresponding to this cleft Hopf Galois extension belongs to the centre of the coinvariant  subalgebra. Moreover, we show that any Ehresmann Schauenburg bialgebroid of this kind is isomorphic to a 2-cocycle twist of the Ehresmann Schauenburg bialgebroid corresponding to a Hopf Galois extension without cocycle, where comodule algebra is an ordinary smash product of the coinvariant subalgebra and the Hopf algebra (i.e. $\C(B\#_{\sigma}H, H)\simeq \C(B\#H, H)^{\tilde{\sigma}}$).
We also study the theory in the case of a Galois object where the base is trivial but without requiring the Hopf algebra to be cocommutative.
\end{abstract}

\maketitle
\tableofcontents
\parskip = .75 ex

\section{Introduction}

The study of principal bundles and groupoids are important in different areas of mathematics and physics. In the area of noncommutative geometry, we are always interested in Hopf Galois extensions, which can be viewed as quantisation of principal bundles.
For any principal bundle, we can construct a gauge groupoid. In the `quantum' case, the `quantum' gauge groupoid can also be constructed for any Hopf Galois extension $B=A^{coH}\subseteq A$ (quantum principal bundles). This kind of `quantum' gauge groupoids $\C(A, H)$ are called Ehresmann Schauenburg bialgebroids.\\

In this paper, we will study cleft Hopf Galois extensions, which can be viewed as the quantisation of trivial principal bundles. Since any cleft Hopf Galois extension is isomorphic to the crossed product $B\#_{\sigma}H$ of the coinvariant subalgebra $B=A^{coH}$ and the Hopf algebra $H$ for a unital cocycle $\sigma: H\otimes H\to B$, so instead of studying the comodule algebra $A$ directly, we can work on the crossed product algebra $B\#_{\sigma}H$.

As it was shown in \cite{D89} that given a cocommutative Hopf algebra $H$, there is a bijective correspondence between the equivalence classes of $H$-cleft Hopf Galois extensions and the second cohomology group $\mathcal{H}^{2}(H, Z(B))$, where $Z(B)$ is the centre of the coinvariant subalgebra $B$.
So in this paper we will mainly consider cleft Hopf Galois extensions with cocommutative Hopf algebra and its second cohomology group $\mathcal{H}^{2}(H, Z(A))$.
Under this special assumption, we can show that there is an invertible normalised 2-cocycle $\tilde{\sigma}$ on the Ehresmann Schauenburg bialgebroid associated to the cleft Hopf Galois extension. Moreover, the Ehresmann Schauenburg bialgebroid is isomorphic to the 2 cocycle twisted algebra $\C(B\#H, H)^{\tilde{\sigma}}$, where $B\#H$ is the smash product of $B$ and $H$. We can also show that if the action of $H$ on $B$ is trivial, then the corresponding Ehresmann Schauenburg bialgebroid associated to the cleft Hopf Galois extension is isomorphic to  $\C(B\#H, H)^{\tilde{\sigma}}$ even if the Hopf algebra $H$ is not cocommutative. In particular, we will study the case of Galois objects for any Hopf algebras. \\

In Section \S2, we will give a brief introduction to Hopf Galois extension, and  in particular the cleft Hopf Galois extension with its properties. Then we will also recall bialgebroids and 2-cocycles of them. In Section \S3, we will first study Ehresmann Schauenburg bialgebroid, then we show under some conditions there is an invertible normailsed 2 cocycle  on $\C(B\#H, H)$, such that $\C(A, H)\simeq\C(B\#_{\sigma}H, H)\simeq \C(B\#H, H)^{\tilde{\sigma}}$. Finally, we will apply the general theory on Galois objects.

\section{Algebraic preliminaries}

In this section we will first recall the definitions of comodule algebras and module algebras of Hopf algebras, then we will study Hopf Galois extensions which can be viewed as the quantisation of principal bundles. In particular, we will recall some properties of cleft Hopf Galois extensions, which can be viewed as trivial noncommutative principal bundles.
We also recall the more general notions of rings and corings over an algebra as well as the associated notion of bialgebroids. In this paper, we will assume all the algebras, comodules and modules are vector spaces over $\mathbb{C}$.

Let $H$ be a Hopf algebra with coproduct $\Delta$, counit $\epsilon$ and antipode $S$. We use the sumless Sweedler notation to denote the image of the coproduct, i.e. $\Delta(h)=\one{h}\otimes \two{h}$ for all $h\in H$. The convolution algebra of the dual space $H':=\Hom(H, \mathbb{C})$ is an unital associative algebra with the product given by $\phi\star\psi(h):=\phi(\one{h})\psi(\two{h})$, for all $\phi, \psi\in H'$.

Given a Hopf algebra $H$, a \textit{left $H$ module algebra} is an algebra $B$, such that it is a left $H$ module. Moreover, it satisfies:
\begin{align}
    h\triangleright(ab)=(\one{h}\triangleright a)(\two{h}\triangleright b),\quad h\triangleright 1=\epsilon(h)1,
\end{align}
for all $a, b\in A$ and $h, g\in H$, where $\triangleright: H\otimes B\to B$ is the action of the left module $B$.\\
Given a left $H$ module algebra $A$, we can define a new algebra $A\#H$, which as a vector space is equal to $A\otimes H$, and the product is given by
\begin{align}\label{equ. smash product}
    (a\#h)(a'\#g)=a(\one{h}\triangleright a')\#\two{h}g,
\end{align}
for all $a, a'\in A$ and $g, h\in H$, where we write $a\#h$ for the tensor product $a\otimes h$. We call this algebra the \textit{smash product} of $A$ and $H$.\\

Dually, given a Hopf algebra $H$, the \textit{right $H$-comodule} is vector space $V$, together with a coaction, which is a linear map $\delta: V\to V\otimes H$ such that
\begin{align}\label{equ. coaction}
    (\id\otimes \Delta)\circ\delta=(\delta\otimes \id)\circ \delta,\qquad (\id\otimes \epsilon)\circ\delta=\id.
\end{align}
We also use the sumless Sweedler notation to denote the image of coaction, i.e. $\delta(v)=\zero{v}\otimes \one{v}$
 for all $v\in V$. The morphism between two right $H$-comodules $V$ and $W$ is a linear map $f: V\to W$, such that
 \begin{align}
     f(\zero{v})\otimes \one{v}=\zero{f(v)}\otimes \one{f(v)}
 \end{align}
for any $v\in V$ . \\

Given two right $H$-comodules $V$ and $W$, the tensor product $V\otimes W$ is also a right $H$-comodule with the coaction given by
\begin{align}
    \delta^{diag}(v\otimes w):=\zero{v}\otimes \zero{w}\otimes \one{v}\one{w}
\end{align}
for all $v\in V$ and $w\in W$.\\

Given a Hopf algebra $H$, a \textit{right $H$-comodule algebra} is an algebra $A$, such that $A$ is a right $H$-comodule, and the comodule map $\delta: A\to A\otimes H$ is an algebra map. The morphism between two comodule algebras is both a comodule map and an algebra map, where $A\ot  H$ has the usual tensor
product algebra structure.

\subsection{Cleft Hopf Galois extensions}\label{sec:chge}

In this section, we will recall the definition of Hopf Galois extensions. In particular, we will study cleft Hopf Galois extensions, which could be thought of as trivial noncommutative  principal bundles.

\begin{defi}
Let $A$ be a comodule algebra of a Hopf algebra $H$, with the coinvariant subalgebra $B:=\big\{b\in A ~|~ \delta (b) = b \ot 1_H \big\} \subseteq A$, the extension $B=A^{coH}\subseteq A$ is called a \textup{Hopf Galois extension}, if the \textup{canonical map}
\begin{align}
    \chi: A\otimes_{B}A\to A\otimes H, \quad a\otimes_{B}a'\mapsto a\zero{a'}\otimes \one{a}
\end{align}
is bijective, where $\otimes_{B}$ is the balanced tensor product over $B$ (i.e. $ab\otimes_{B}a'=a\otimes_{B}ba'$ for all $a, a'\in A$ and $b\in B$).
\end{defi}

In the following, we will always assume that for any Hopf Galois extension $B=A^{coH}\subseteq A$, $A$ is a faithful flat left $B$ module.

Given a Hopf Galois extension $B=A^{coH}\subseteq A$, the \textit{translation map} is
\begin{align}\label{equ. canonical map}
    \tau:=\chi^{-1}|_{1\otimes H}:  H\to A\otimes_{B}A,\qquad h\mapsto \tuno{h}\otimes_{B}\tdue{h}.
\end{align}
Since the canonical map $\chi$ is left $B$ linear, then the inverse of the canonical map can be determined by the translation map. It is shown in \cite[Prop. 3.6]{brz-tr} and \cite[Lemma 34.4]{BW} that the translation map of a Hopf Galois extension satisfy the following properties:

\begin{align}\label{equ. translation map 1}
  \tuno{h} \ot_B \zero{\tdue{h}} \ot \one{\tdue{h}} &= \,\tuno{\one{h}} \ot_B \tdue{\one{h}} \ot
\two{h} ~,\\
\label{equ. translation map 2}
~~ \tuno{\two{h}}  \ot_B \tdue{\two{h}} \ot S(\one{h}) &= \zero{\tuno{h}} \ot_B {\tdue{h}}  \ot \one{\tuno{h}},\\
\label{equ. translation map 3}
\tuno{h}\zero{\tdue{h}}\ot \one{\tdue{h}} &= 1_{A} \ot h ~,\\
\label{equ. translation map 4}
    \zero{a}\tuno{\one{a}}\ot_{B}\tdue{\one{a}} &= 1_{A} \ot_{B}a ~,
\end{align}
for all $h\in H$ and $a\in A$.

Now we recall cleft Hopf Galois extensions.

\begin{defi}
A Hopf--Galois extension
$B=A^{co \, H}\subseteq A$ is \textup{cleft} if there is a convolution invertible right $H$ comodule map $\gamma: H\to A$.
\end{defi}

A Hopf Galois extension has \textit{normal basis property}, if $A\simeq B\otimes H$ as left $B$-modules and right $H$-comodules.

\begin{defi}\label{definition. measure}
Given a Hopf algebra $H$ and an algebra $B$, we call  \textup{$H$ measures $B$}, if there is a linear map $H\otimes B\to B$, given by $h\otimes b\mapsto h\triangleright b$, such that
\begin{itemize}
    \item $h\triangleright 1=\epsilon(h)1$,
    \item $h\triangleright(bb')=(\one{h}\triangleright b)(\two{h} \triangleright b')$,
\end{itemize}
for all $h\in H$ and $b, b'\in A$.
\end{defi}

It was shown in \cite{DT86} that we can construct an algebra by $H$ and $A$, if $H$ measures $A$.

\begin{lem}\label{lemma. twisted smash product}
Let $H$ be a Hopf algebra, $B$ be an algebra and $H$ measures $B$. If there is a convolution invertible linear map $\sigma: H\otimes H\to B$ (i.e. there is a linear map $\sigma^{-1}: H\otimes H\to B$, such that $\sigma(\one{h}, \one{h'})\sigma^{-1}(\two{h}, \two{h'})=\epsilon(h)\epsilon(h')1=\sigma^{-1}(\one{h}, \one{h'})\sigma(\two{h}, \two{h'})$ for all $h, h'\in H$), such that
\begin{itemize}
    \item [(1)]$1\triangleright b=b$,
    \item [(2)]$h\triangleright(k\triangleright b)=\sigma(\one{h}, \one{k})(\two{h}\two{k}\triangleright b)\sigma^{-1}(\three{h}, \three{k})$,
    \item [(3)] $\sigma(h, 1)=\sigma(1, h)=\epsilon(h)1$,
    \item [(4)] $(\one{h}\triangleright\sigma(\one{k},\one{m}))\sigma(\two{h}, \two{k}\two{m})=\sigma(\one{h}, \one{k})\sigma(\two{h}\two{k}, m)$,
\end{itemize}
for all $h, k, m\in H$ and $b\in B$. Then there is an algebra structure on $B\#_{\sigma} H$, which is equal to $B\otimes H$ as vector space, with the product given by
\begin{align}\label{equation. product of twisted crossed product}
    (b\#_{\sigma} h)(b'\#_{\sigma} h')=b(\one{h}\triangleright b')\sigma(\two{h}, \one{h'})\#_{\sigma} \three{h}\two{h'},
\end{align}
for all $h, h'\in H$ and $b, b'\in B$. Here we have written $b\#_{\sigma} h$ for the tensor product $b\otimes h$. Conversely, assume $H$ measures $B$ and $\sigma: H\otimes H\to B$ is a convolution invertible linear map, if the vector space $B\#_{\sigma} H$ (equal to $B\otimes H$ as vector space) with a product defined by (\ref{equation. product of twisted crossed product}) is an associative algebra with identity element $1\#_{\sigma}1$, then we have the four conditions above.
\end{lem}

The algebra $B\#_{\sigma}H$ given above is called the \textit{crossed product} of $B$ and $H$.
Let $H$ measures $B$, we call $\sigma$ an \textit{unital cocycle}, if $\sigma(h, 1)=\sigma(1, h)=\epsilon(h)1$ and
\begin{align}\label{equation. cocycle condition for H}
    (\one{h}\triangleright\sigma(\one{k},\one{m}))\sigma(\two{h}, \two{k}\two{m})=\sigma(\one{h}, \one{k})\sigma(\two{h}\two{k}, m),
\end{align}
for any all $h, k, m\in H$.

If $\sigma$ is trivial (i.e. $\sigma(h, h')=\epsilon(hh')1$), then from Lemma \ref{lemma. twisted smash product} we know that $B$ is a left module algebra of $H$ and $B\#H$ is the smash product of $B$ and $H$ given by (\ref{equ. smash product}).

It is known (cf. \cite{mont}, Proposition 7.2.7) that:
\begin{prop}\label{proposition. unital twist}
Let $H$ measures $B$ with a convolution invertible linear map $\sigma: H\otimes H\to B$, such that $B\#_{\sigma}H$ is an unital associative algebra, then we have the following properties of $\sigma$:
\begin{itemize}
    \item [(1)] $\sigma^{-1}(\one{h}, \one{k}\one{m})(\two{h}\triangleright \sigma^{-1}(\two{k}, \two{m}))=\sigma^{-1}(\one{h}\one{k}, m)\sigma^{-1}(\two{h}, \two{k})$,
    \item[(2)] $h\triangleright \sigma(k, m)=\sigma(\one{h}, \one{k})\sigma(\two{h}\two{k}, \one{m})\sigma^{-1}(\three{h}, \three{k}\two{m})$,
    \item[(3)] $h\triangleright \sigma^{-1}(k, m)=\sigma(\one{h}, \one{k}\one{m})\sigma^{-1}(\two{h}\two{k}, \two{m})\sigma^{-1}(\three{h}, \three{k})$,
    \item[(4)] $(\one{h}\triangleright \sigma^{-1}(S(\four{h}), \five{h}))\sigma(\two{h}, S(\three{h}))=\epsilon(h)1$.
\end{itemize}
\begin{proof}
We can see that (1) and (2) can be derived from Lemma \ref{lemma. twisted smash product}, and (3) can be derived from (1). Now let's check (4):
\begin{align*}
    &(\one{h}\triangleright \sigma^{-1}(S(\four{h}), \five{h}))\sigma(\two{h}, \three{h})\\
    &=\sigma(\one{h}, S(\eight{h})\nine{h})\sigma^{-1}(\two{h}S(\seven{h}), \ten{h})\sigma^{-1}(\three{h}, S(\six{h}))\sigma(\four{h}, S(\five{h}))\\
    &=\sigma(\one{h}, S(\four{h})\five{h})\sigma^{-1}(\two{h}S(\three{h}), \six{h})\\
    &=\epsilon(h)1,
\end{align*}
for any $h\in H$.
\end{proof}

\end{prop}

We can see that $B=(B\#_{\sigma}H)^{coH}\subseteq B\#_{\sigma}H$ is a Hopf Galois extension with the coaction $\delta(b\#_{\sigma}h):=b\#_{\sigma}\one{h}\otimes \two{h}$ for all $b\#_{\sigma}h\in B\#_{\sigma}H$. The inverse of its canonical map is given by \cite{mont}
\begin{align}\label{equ. inverse of canonical map of cleft extension}
    \chi^{-1}(b\#_{\sigma}g\otimes h)=(b\#_{\sigma}g)(\sigma^{-1}(S(\two{h}), \three{h})\#_{\sigma} S(\one{h}))\otimes_{B}1\#_{\sigma}\four{h},
\end{align}
for all $b\in B$ and $g, h\in H$. Indeed, we can check that
\begin{align*}
    &\chi\big((b\#_{\sigma}g)(\sigma^{-1}(S(\two{h}), \three{h})\#_{\sigma}S(\one{h}))\otimes_{B}1\#_{\sigma}\four{h}\big)\\
    &=(b\#_{\sigma}g)\big(\sigma^{-1}(S(\three{h}), \four{h})\sigma(S(\two{h}), \five{h})\#_{\sigma}S(\one{h})\six{h}\big)\otimes \seven{h}\\
    &=b\#_{\sigma}g\otimes h.
\end{align*}
We also have
\begin{align*}
&\chi^{-1}(\chi(b\#_{\sigma}g\otimes_{B}b'\#_{\sigma}h))\\
&=(b\#_{\sigma}g)(b'\#_{\sigma}\one{h})(\sigma^{-1}(S(\three{h}), \four{h})\#_{\sigma}S(\two{h}))\otimes_{B} 1\#_{\sigma}\five{h}\\
&=(b\#_{\sigma}g)\big(b'(\one{h}\triangleright \sigma^{-1}(S(\six{h}), \seven{h}))\sigma(\two{h}, S(\five{h}))\#_{\sigma}\three{h}S(\four{h})\big)\otimes_{B}1\#_{\sigma}\eight{h}\\
&=(b\#_{\sigma}g)(b'\#_{\sigma}1)\otimes_{B}1\#_{\sigma}h\\
&=b\#_{\sigma}g\otimes_{B}b'\#_{\sigma}h,
\end{align*}
for all $b, b'\in B$ and $g, h\in H$, where the third step uses Proposition \ref{proposition. unital twist}.
It is known (cf. \cite{mont}, Theorem 8.2.4) that:
\begin{thm}\label{theorem. cleft extension}
Let $B=A^{co \, H}\subseteq A$ be a Hopf Galois extension. Then the following are equivalent:
\begin{itemize}
    \item $B=A^{co \, H}\subseteq A$ is cleft.
    \item $B=A^{co \, H}\subseteq A$ has normal basis property.
    \item $A\simeq B\#_{\sigma}H$ as left $B$-modules and right $H$-comodule algebras.
\end{itemize}
\end{thm}

\subsection{Bialgebroids}

Here we also give an introduction of bialgebroids (cf. \cite{BW}, \cite{Boehm}).
For an algebra $B$, a {\em $B$-ring} is a triple
$(A,\mu,\eta)$. Here $A$ is a $B$-bimodule and $\mu:A\ot_ {B} A \to A$ and
$\eta:B\to A$ are $B$-bimodule maps, satisfying the associativity and unit
conditions
\begin{equation}\label{eq:cat.associ}
\mu\circ(\mu\otimes_{B} \id_{A})=\mu\circ (\id_{A}\otimes_{B} \mu)\qquad \textrm{and}\qquad
\mu\circ(\eta \otimes_{B} \id_{A})=\id_{A}=\mu\circ (\id_{A}\otimes_{B} \eta).
\end{equation}
A {\em morphism of $B$-rings} $f:(A,\mu,\eta)\to (A',\mu',\eta')$ is an
$B$-bimodule map $f:A \to A'$, such that
$f\circ \mu=\mu'\circ(f\ot_{B} f)$ and $f\circ \eta=\eta'$.

\begin{rem}\label{rring}
Let $B$ and $A$ be algebras, if there is an algebra map $\eta: B\to A$, then $A$ is a $B$-bimodule with $b\triangleright a\triangleleft b'=\eta(b)a\eta(b')$. Moreover, $A$ is a $B$-ring with the product obtained from the universality of the coequaliser $A \ot A \to A\ot_B A$ which identifies an element $ a b \ot a'$ with $ a \ot b a'$.
\end{rem}

For an algebra $B$, a {\em $B$-coring} is a
triple $(C,\Delta,\epsilon)$. Here $C$ is an $B$-bimodule and $\Delta:C\to
C\ot_{B} C$ and $\epsilon: C\to B$ are $B$-bimodule maps, satisfying the
coassociativity and counit conditions.
\begin{equation}\label{eq:cat.coassoci}
(\Delta\otimes_{B} \id_{C})\circ \Delta=(\id_{C}\otimes_{B} \Delta)\circ \Delta
\qquad \textrm{and}\qquad
(\epsilon \otimes_{B} \id_{C})\circ \Delta=\id_{C}=(\id_{C}\otimes_{B} \epsilon)\circ \Delta.
\end{equation}
A {\em morphism of $B$-corings} $f:(C,\Delta,\epsilon)\to
(C',\Delta',\epsilon')$ is a $B$-bimodule map $f:C \to C'$, such that
$\Delta'\circ f=(f\ot_{B} f)\circ \Delta$ and
$\epsilon' \circ f =\epsilon$.

\begin{defi}\label{def:right.bgd}
Let $B$ be an algebra.
A {\em left $B$-bialgebroid} $\cL$ consists of a $B\otimes B^{op}$-ring $\cL$ with the unit $\eta: B\otimes B^{op}\to \cL$. The restrictions of $\eta$
$$
s := \eta ( \, \cdot \, \ot_B 1_B ) : B \to \cL \quad \mbox{and} \quad t := \eta ( 1_B \ot_B  \, \cdot \, ) : B^{op} \to \cL
$$
are called source and target map, with their ranges commute in $B$.

Moreover, $\cL$ is a $B$-coring $(\cL,\Delta,\epsilon)$ on the same vector space
$\cL$. They are subject to the following compatibility axioms.
\begin{itemize}
\item[(i)] The bimodule structure in the  $B$-coring $(\cL,\Delta,\epsilon)$ is
  related to the $B\ot B^{op}$-ring $\cL$ via
\begin{equation}\label{eq:rbgd.bimod}
b\triangleright a\triangleleft b':= s(b) t(b')a,\quad \textrm{for }b,b'\in B,\ a\in \cL.
\end{equation}
\item[(ii)] Considering $\cL$ as an $B$-bimodule as in \eqref{eq:rbgd.bimod},
  the coproduct $\Delta$ corestricts to an algebra map from $\cL$ to
\begin{equation}\label{eq:Tak.prod}
\cL\times_{B} \cL :=\{\ \sum_i a_i \ot_{B} a'_i\ |\ \sum_i a_it(b) \ot_{B} a'_i=
\sum_i a_i \ot_{B}  a'_i  s(b),\quad \forall b\in B\ \},
\end{equation}
where $\cL\times_B \cL$ is an algebra via factorwise multiplication.
\item[(iii)] The counit $\epsilon$ is a left character on the $B$-ring $(\cL,s)$:
\begin{itemize}
\item[(1)] $\epsilon(1_{\cL})=1_{B}$.
\item[(2)] $\epsilon( s(b)a)=b\epsilon(a)$,
\item[(3)] $\epsilon(as(\epsilon(a')))=\epsilon(aa')=\epsilon(at (\epsilon(a')))$,
\end{itemize}
for all $a, a'\in \cL$ and $b\in B$.
\end{itemize}
\end{defi}

Morphisms between left $B$ bialgebroids are $B$-coring maps which are also algebra maps.
Given a left $B$-bialgebroid $\cL$, then there is an algebra structure on ${}_{B}\Hom_{B}(\cL\otimes_{B\otimes B^{op}}\cL, B)$ with the (convolution) product given by
\begin{align}
    f\star g(a, a'):=f(\one{a}, \one{a'})g(\two{a}, \two{a'}),
\end{align}
for all $a, a'\in \cL$ and $f, g\in {}_{B}\Hom_{B}(\cL\otimes_{B\otimes B^{op}}\cL, B)$.
The unit of this algebra is $\tilde{\epsilon}: a\otimes a'\mapsto \epsilon(aa')$. Moreover, the $B$-bimodule structure on $\cL\otimes_{B\otimes B^{op}}\cL$ is give by $b\triangleright(a\otimes a')\triangleleft b'=s(b)t(b')a\otimes a'$, for all $b, b'\in B$, and the balanced tensor product $\cL\otimes_{B\otimes B^{op}}\cL$ is induced by the algebra map $\eta: B\otimes B^{op}\to \cL$.

\begin{defi}
Let $\cL$ be a left $B$-bialgebroid, an \textup{invertible normalised 2-cocycle} on
$\cL$ is a convolution invertible element $\tilde{\sigma}\in {}_{B}\Hom_{B}(\cL\otimes_{B\otimes B^{op}}\cL, B)$, such that
\begin{itemize}
    \item [(1)] $\tilde{\sigma}(s(b)t(b')a, a')=b\tilde{\sigma}(a, a')b'$ (bilinearity),
    \item[(2)] $\tilde{\sigma}(a, s(\tilde{\sigma}(\one{a'}, \one{a''}))\two{a'}\two{a''})=\tilde{\sigma}(s(\tilde{\sigma}(\one{a}, \one{a'}))\two{a}\two{a'}, a'')$ (cocycle condition).
    \item[(3)] $\title{\sigma}(1_{\cL}, a)=\epsilon(a)=\title{\sigma}(a, 1_{\cL})$ (normalisation),
\end{itemize}
for all $a, a', a''\in \cL$ and $b, b'\in B$.
\end{defi}

It is known in \cite{Boehm}, we can define a new left bialgebroid by an invertible normalised 2-cocycle.

\begin{prop}\label{proposition. 2 cocycle twist}
Let $\cL$ be a left $B$-bialgebroid and $\tilde{\sigma}\in {}_{B}\Hom_{B}(\cL\otimes_{B\otimes B^{op}}\cL, B)$ be an invertible normalised 2-cocycle, with inverse $\tilde{\sigma}^{-1}$, then the $B$-coring $\cL$ with the twisted product
\begin{align}
    a\cdot_{\tilde{\sigma}} a':=s(\tilde{\sigma}(\one{a}, \one{a'}))t(\tilde{\sigma}^{-1}(\three{a}, \three{a'}))\two{a}\two{a'},
\end{align}
for all $a, a'\in \cL$, constitute a left $B$-bialgebroid $\cL^{\tilde{\sigma}}$.
\end{prop}

The proposition above can also apply to Hopf algebras, which are left bialgebroid over $\mathbb{C}$.

\section{Ehresmann-Schauenburg bialgebroids}

Ehresmann Schauenburg Bialgebroids can be viewed as the quantisation of Gauge groupoids associated to principal bundles. Here we will first give the definition of Ehresmann Schauenburg Bialgebroids \cite[\S 34.13]{BW}.

\begin{defi}\label{def:ec}
Let $B=A^{co H}\subseteq A$ be a Hopf Galois extension such that $A$ is a faithful flat left $B$-module, the $B$-bimodule

\begin{equation}\label{ec1}
\mathcal{C}(A, H) :=\{a\otimes \tilde{a}\in A\otimes A\quad|\quad \zero{a}\otimes \tau(\one{a})\tilde{a}=a\otimes \tilde{a}\otimes_B 1\}.
\end{equation}
 is a $B$-coring with the coring product and counit given by
\beq\label{copro}
\Delta(a\otimes \tilde{a})=\zero{a}\otimes \tau(\one{a})\otimes\tilde{a},
\eeq
\beq\label{counit}
\epsilon(a\otimes \tilde{a})=a\tilde{a}.
\eeq
Moreover, $\C(A, H)$ is a $B\otimes B^{op}$-ring with the product given by

\beq\label{pro}
(a\otimes \tilde{a})\bullet_{\mathcal{C}}({a}^{'}\otimes \tilde{a}^{'})=aa^{'}\otimes\tilde{a}^{'}\tilde{a} ,
\eeq
for all $a\otimes \tilde{a}$, $a^{'}\otimes \tilde{a}^{'} \in \mathcal{C}$. The source and target maps are

\beq\label{sourcemap}
s(b)=b\otimes 1,
\eeq
\beq\label{targetmap}
t(b)=1\otimes b.
\eeq
All of the structures given above form a left $B$-bialgebroid, which is called \textup{Ehresmann Schauenburg bialgebroid}.
\end{defi}

It was shown in \cite{BW}, \cite{HL20} that the $B$-bimodule $\C(A, H)$ is isomorphic to the $B$-bimodule of coinvariant elements, that is

\beq
\C(A, H)\simeq(A\ot A)^{coH} := \{a\ot  \tilde{a}\in A\ot  A \, \quad|\quad \zero{a}\ot  \zero{\tilde{a}}\ot  \one{a}\one{\tilde{a}}=a\ot  \tilde{a}\ot  1_H \}. \label{ec2}
\eeq

\subsection{2-cocycles on Ehrenmann-Schauenburg bialgebroids}

Let $B=A^{coH}\subseteq A$ be a cleft Hopf Galois extension, then by Theorem \ref{theorem. cleft extension} we know  $\C(A, H)\simeq \C(B\#_{\sigma} H, H)$. In particular, $B=(B\# H)^{coH}\subseteq B\# H$ is a cleft Hopf Galois extension. From (\ref{ec2}) we can see
\begin{align}
    \C(B\# H, H)=\{b\#\one{h}\otimes b'\# S(\two{h})\quad|\quad \forall  b, b'\in B \quad \textup{and}\quad  \forall h\in H\}.
\end{align}
We can also see the coproduct is given by
\begin{align}
    \Delta(b\#\one{h}\otimes b'\# S(\two{h}))=b\#\one{h}\otimes 1\# S(\two{h})\otimes_{B} 1\#\three{h}\otimes b'\# S(\four{h}),
\end{align}
since from (\ref{equ. inverse of canonical map of cleft extension}) we know the translation map $\tau: H\to B\# H\otimes_{B}B\# H$ is
\begin{align}
    \tau(h)=1\#S(\one{h})\otimes_{B}1\#\two{h},
\end{align}
for all $h\in H$.
\begin{lem}\label{lemma. construction of 2 cocycle}
For a crossed product $B\#_{\sigma} H$, with $H$ being cocommutative and the image of $\sigma$ belonging to the centre of $B$, then the linear map $\tilde{\sigma}\in {}_{B}\Hom_{B}(\C(B\# H, H)\otimes_{B\otimes B^{op}}\C(B\# H, H), B)$ given by
\begin{align}\label{equ. 2-cocycle on bialgebroid}
    \tilde{\sigma}(b\#\one{h}\otimes b'\# S(\two{h}), c\#\one{g}\otimes c'\# S(\two{g})):=b(\one{h}\triangleright c)\sigma(\two{h}, \one{g})((\three{h}\two{g})\triangleright c')(\four{h}\triangleright b'),
\end{align}
for all $b\#\one{h}\otimes b'\# S(\two{h})$, $c\#\one{g}\otimes c'\# S(\two{g})\in \C(B\# H, H)$
is  an  invertible  normalised 2-cocycle.
\end{lem}
\begin{proof}
Since $H$ is cocommutative and the image of $\sigma$ belongs to the centre of $B$, then by Lemma \ref{lemma. twisted smash product},  $B$ is a left $H$-module algebra, which ensures the smash product $B\#H$ is well defined.
Let $X=b\#\one{h}\otimes b'\# S(\two{h})$, $Y=c\#\one{g}\otimes c'\# S(\two{g})$ and $Z=d\#\one{k}\otimes d'\# S(\two{k})$ be arbitrary three elements in $\C(B\#H, H)$.
First we can show this cocycle is well defined over the balanced tensor over $B\otimes B^{op}$:
On the one hand we have
\begin{align*}
    \tilde{\sigma}(X\eta(d\otimes d'), Y)) &=\tilde{\sigma}((b\#\one{h}\otimes b'\# S(\two{h}))\eta(d\otimes d'), c\#\one{g}\otimes c'\# S(\two{g}))\\
     &=b(\one{h}\triangleright d)(\two{h}\triangleright c)\sigma(\three{h}, \one{g})((\four{h}\two{g})\triangleright c')(\five{h}\triangleright(d'b')),
\end{align*}
for all $d, d'\in B$; On the other hand
\begin{align*}
     \tilde{\sigma}(X, \eta(d\otimes d')Y))&=\tilde{\sigma}((b\#\one{h}\otimes b'\# S(\two{h})), \eta(d\otimes d')c\#\one{g}\otimes c'\# S(\two{g}))\\
     &=\tilde{\sigma}((b\#\one{h}\otimes b'\# S(\two{h})), dc\#\one{g}\otimes c'(S(\three{g})\triangleright d')\# S(\two{g}))\\
     &=b(\one{h}\triangleright(dc))\sigma(\two{h}, \one{g})((\three{h}\two{g})\triangleright(c'(S(\three{g})\triangleright d')))(\four{h}\triangleright b')\\
     &=b(\one{h}\triangleright d)(\two{h}\triangleright c)\sigma(\three{h}, \one{g})((\four{h}\two{g})\triangleright c')(\five{h}\triangleright(d'b')).
\end{align*}
The inverse $\tilde{\sigma}^{-1}$ is given by
\begin{align}
     \tilde{\sigma}^{-1}(X, Y):=b(\one{h}\triangleright c)\sigma^{-1}(\two{h}, \one{g})((\three{h}\two{g})\triangleright c')(\four{h}\triangleright b').
\end{align}
We can see that
\begin{align*}
 \tilde{\sigma}^{-1}\star  \tilde{\sigma}(X, Y)&=\tilde{\sigma}^{-1}\star  \tilde{\sigma}(b\#\one{h}\otimes b'\# S(\two{h}), c\#\one{g}\otimes c'\# S(\two{g}))\\
 &=b(\one{h}\triangleright c)\sigma^{-1}(\two{h}, \one{g})\sigma(\three{h}, \two{g})((\four{h}\three{g})\triangleright c')(\five{h}\triangleright b')\\
 &=\epsilon((b\#\one{h}\otimes b'\# S(\two{h}))(c\#\one{g}\otimes c'\# S(\two{g})))\\
 &=\tilde{\epsilon}(X, Y),
\end{align*}
where $\tilde{\epsilon}$ is the unit in the algebra ${}_{B}\Hom_{B}(\cL\otimes_{B\otimes B^{op}}\cL, B)$. Similarly, we can also see $\tilde{\sigma}\star  \tilde{\sigma}^{-1}=\tilde{\epsilon}$.

It is clear that $\tilde{\sigma}$ is left $B$-linear, we can also show that $\tilde{\sigma}$ is also right $B$-linear:
\begin{align*}
    \tilde{\sigma}(t(b'')X, Y)&=\tilde{\sigma}(t(b'')(b\#\one{h}\otimes b'\# S(\two{h})), c\#\one{g}\otimes c'\# S(\two{g}))\\
    &=\tilde{\sigma}(b\#\one{h}\otimes b'(S(\three{h})\triangleright b'')\# S(\two{h}), c\#\one{g}\otimes c'\# S(\two{g}))\\
    &=b(\one{h}\triangleright c)\sigma(\two{h}, \one{g})((\three{h}\two{g})\triangleright c')(\four{h}\triangleright (b'(S(\five{h})\triangleright b'')))\\
    &=b(\one{h}\triangleright c)\sigma(\two{h}, \one{g})((\three{h}\two{g})\triangleright c')(\four{h}\triangleright b')b''\\
    &=\tilde{\sigma}(X, Y)b'',
\end{align*}
for all $b''\in B$. Now let's show the cocycle condition of $\tilde{\sigma}$.

On the one hand we have:
\begin{align*}
    &\tilde{\sigma}(X, s(\tilde{\sigma}(\one{Y}, \one{Z}))\two{Y}\two{Z})\\
    &=\tilde{\sigma}(b\#\one{h}\otimes b'\# S(\two{h}), c(\one{g}\triangleright d)\sigma(\two{g}, \one{k})\#\three{g}\two{k}\otimes d'(S(\four{k})\triangleright c')\#S(\three{k})S(\four{g})))\\
    &=b\Big(\one{h}\triangleright \big(c(\one{g}\triangleright d)\sigma(\two{g}, \one{k})\big)\Big)\sigma(\two{h}, \three{g}\two{k})\big((\three{h}\four{g}\three{k})\triangleright (d'(S(\four{k})\triangleright c'))\big)(\four{h}\triangleright b').
\end{align*}
On the other hand,
\begin{align*}
    &\tilde{\sigma}(s(\tilde{\sigma}(\one{X}, \one{Y}))\two{X}\two{Y}, Z)\\
    &=\tilde{\sigma}(b(\one{h}\triangleright c)\sigma(\two{h}, \one{g})\#\three{h}\two{g}\otimes c' (S(\four{g})\triangleright b')\# S(\three{g})S(\four{h}), d\#\one{k}\otimes d'\#S(\two{k}))\\
    &=b(\one{h}\triangleright c)\sigma(\two{h}, \one{g})((\three{h}\two{g})\triangleright d)\sigma(\four{h}\three{g}, \one{k})((\five{h}\four{g}\two{k})\triangleright d')\big((\six{h}\five{g})\triangleright(c'(S(\six{g})\triangleright b'))\big).
\end{align*}
Compare the results on both hand sides, it is sufficient to show
\begin{align*}
    \one{h}\triangleright((\one{g}\triangleright d)\sigma(\two{g}, \one{k}))\sigma(\two{h}, \three{g}\two{k})=\sigma(\one{h}, \one{g})((\two{h}\two{g})\triangleright d)\sigma(\three{h}\three{g}, k).
\end{align*}
We can see the left hand side is
\begin{align*}
     &\one{h}\triangleright((\one{g}\triangleright d)\sigma(\two{g}, \one{k}))\sigma(\two{h}, \three{g}\two{k})\\
     &=(\one{h}\triangleright (\one{g}\triangleright d))(\two{h}\triangleright\sigma(\two{g}, \one{k}))\sigma(\three{h}, \three{g}\two{k})\\
     &=(\one{h}\triangleright (\one{g}\triangleright d))\sigma(\two{h}, \two{g})\sigma(\three{h}\three{g}, k)\\
     &=\sigma(\one{h}, \one{g})((\two{h}\two{g})\triangleright d)\sigma(\three{h}\three{g}, k),
\end{align*}
where the second step uses (\ref{equation. cocycle condition for H}), and the last step uses Lemma \ref{lemma. twisted smash product}. Finally, we can show the normalisation condition:
\begin{align*}
    &\tilde{\sigma}(X, 1)=b(\one{h}\triangleright b')=\epsilon(X)\\
    &\tilde{\sigma}(1, X)=b(\one{h}\triangleright b')=\epsilon(X).
\end{align*}

\end{proof}

\begin{lem}\label{lemma. coring map}
For a crossed product $B\#_{\sigma} H$, with $H$ being cocommutative and the image of $\sigma$ belonging to the centre of $B$, then the map $\phi: \C(B\#H, H)\to \C(B\#_{\sigma}H, H)$ is a $B$-coring map, which is given by
\begin{align}\label{equ. isomorphic map between bialgebroid}
   \phi( b\#\one{h}\otimes b'\# S(\two{h})):=b\#_{\sigma}\one{h}\otimes b'\sigma^{-1}(S(\three{h}), \four{h})\#_{\sigma}S(\two{h})
\end{align}
for any $b\#\one{h}\otimes b'\# S(\two{h})\in \C(B\#H, H)$.
\end{lem}
\begin{proof}
First we check $\epsilon=\epsilon^{\sigma}\circ\phi$, where $\epsilon^{\sigma}$ is the counit of $\C(B\#_{\sigma}H, H)$.
Let $X=b\#\one{h}\otimes b'\# S(\two{h})\in\C(B\#H, H)$, then
\begin{align*}
    \epsilon^{\sigma}(\phi(X))&=(b\#_{\sigma}\one{h})(b'\sigma^{-1}(S(\three{h}), \four{h})\#_{\sigma}S(\two{h}))\\
    &=b(\one{h}\triangleright b')(\two{h}\triangleright \sigma^{-1}(S(\five{h}), \six{h}))\sigma(\three{h}, S(\four{h}))\#_{\sigma} 1\\
    &=b(h\triangleright b')\#_{\sigma} 1\\
    &=\epsilon(b\#\one{h}\otimes b'\# S(\two{h})),
\end{align*}
where in the third step we use Proposition \ref{proposition. unital twist}. Here we always identify $B$ with its image in $B\#H$ and $B\#_{\sigma}H$ by $b\mapsto b\#1$ and $b\mapsto b\#_{\sigma}1$ respectively. We can see $\phi$ is left $B$-module map. We can also check it is right $B$-linear:
\begin{align*}
    \phi(X\triangleleft b'')=&\phi(b\#\one{h}\otimes b' (S(\three{h})\triangleright b'')\# S(\two{h}))\\
    =&b\#_{\sigma}\one{h}\otimes b' (S(\five{h})\triangleright b'')\sigma^{-1}(S(\three{h}), \four{h})\#_{\sigma} S(\two{h})\\
    =&b\#_{\sigma}\one{h}\otimes b' \sigma^{-1}(S(\four{h}), \five{h})(S(\three{h})\triangleright b'')\#_{\sigma} S(\two{h})\\
    =&\phi(X)\triangleleft b'',
\end{align*}
for all $b''\in B$, where the third step uses the fact that $H$ is cocommutative and the image of $\sigma$ belongs to the centre of $B$. Recall that the translation map of the Hopf Galois extension $B=(B\#_{\sigma}H)^{coH}\subseteq B\#_{\sigma}H$ is given by
\begin{align*}
    \tau(h)=\sigma^{-1}(S(\two{h}), \three{h})\#_{\sigma}S(\one{h})\otimes_{B}1\#_{\sigma}\four{h},
\end{align*}
for all $h\in H$. So we have
\begin{align*}
    &(\phi\otimes_{B}\phi)(\Delta(X))\\
    &=(\phi\otimes_{B}\phi)(b\#\one{h}\otimes 1\# S(\two{h})\otimes_{B} 1\#\three{h}\otimes b'\# S(\four{h}))\\
    &=b\#_{\sigma}\one{h}\otimes \sigma^{-1}(S(\three{h}), \four{h})\#_{\sigma}S(\two{h})\otimes_{B}1\#_{\sigma}\five{h}\otimes b'\sigma^{-1}(S(\seven{h}), \eight{h})\#_{\sigma}S(\six{h})\\
    &=\Delta^{\sigma}(b\#_{\sigma}\one{h}\otimes b'\sigma^{-1}(S(\three{h}), \four{h})\#_{\sigma}S(\two{h}))\\
    &=\Delta^{\sigma}(\phi(X)),
\end{align*}
where $\Delta^{\sigma}$ is the coproduct of $\C(B\#_{\sigma}H, H)$.
\end{proof}

\begin{thm}\label{theorem. 2-cocycle twist}
For a crossed product $B\#_{\sigma} H$, with $H$ being cocommutative and the image of $\sigma$ belonging to the centre of $B$, then there is an invertible normalised 2-cocycle $\tilde{\sigma}$ on $\C(B\#H, H)$, such that
$\phi: \C(B\#H, H)^{\tilde{\sigma}}\to \C(B\#_{\sigma}H, H)$ is an isomorphism of left $B$-bialgebroids, where $\tilde{\sigma}$ is given by (\ref{equ. 2-cocycle on bialgebroid}) and $\phi$ is given by (\ref{equ. isomorphic map between bialgebroid}).
\end{thm}
\begin{proof}
Since $\phi$ is a coring map, so we only need to show $\phi$ is an algebra map. Let $X=b\#\one{h}\otimes b'\# S(\two{h})$, $Y=c\#\one{g}\otimes c'\# S(\two{g})\in \C(B\#H, H)$. On the one hand,
\begin{align*}
    &\phi(X\cdot_{\tilde{\sigma}}Y)\\
    &=\phi\Big(\tilde{\sigma}(b\#\one{h}\otimes 1\# S(\two{h}), c\#\one{g}\otimes 1\# S(\two{g}))\#\three{h}\three{g}\otimes\\ &\big((S(\five{g})S(\five{h}))\triangleright \tilde{\sigma}^{-1}(1\#\six{h}\otimes b'\# S(\seven{h}), 1\#\six{g}\otimes c'\# S(\seven{g}))\big)\#S(\four{g})S(\four{h})\Big)\\
    &=\phi\Big(b(\one{h}\triangleright c)\sigma(\two{h}, \one{g})\#\three{h}\two{g}\otimes\\
    &(S(\four{g})S(\five{h}))\triangleright\big(\sigma^{-1}(\six{h}, \five{g})((\seven{h}\six{g})\triangleright c')(\eight{h}\triangleright b')) \big)\#S(\three{g})S(\four{h})\Big)\\
    &=b(\one{h}\triangleright c)\sigma(\two{h}, \one{g})\#_{\sigma}\three{h}\two{g}\otimes(S(\six{g})S(\seven{h}))\triangleright\big(\sigma^{-1}(\eight{h}, \seven{g})((\nine{h}\eight{g})\triangleright c')(\ten{h}\triangleright b')) \big)\\
    &\sigma^{-1}(S(\four{g})S(\five{h}), \six{h}\five{g}))\#_{\sigma}S(\three{g})S(\four{h})\\
    &=b(\one{h}\triangleright c)\sigma(\two{h}, \one{g})\#_{\sigma}\three{h}\two{g}\otimes\\
    &(S(\six{g})S(\seven{h}))\triangleright\sigma^{-1}(\eight{h}, \seven{g})c'(S(\eight{g})\triangleright b')\sigma^{-1}(S(\four{g})S(\five{h}), \six{h}\five{g}))\#_{\sigma}S(\three{g})S(\four{h})\\
    &=b(\one{h}\triangleright c)\sigma(\two{h}, \one{g})\#_{\sigma}\three{h}\two{g}\otimes\sigma(S(\six{g})S(\seven{h}), \eight{h}\seven{g})\sigma^{-1}(S(\eight{g})S(\nine{h})\ten{h}, \nine{g})\\
    &\sigma^{-1}(S(\ten{g})S(\eleven{h}), \twelve{h})c'(S(\eleven{g})\triangleright b')\sigma^{-1}(S(\four{g})S(\five{h}), \six{h}\five{g}))\#_{\sigma}S(\three{g})S(\four{h})\\
    &=b(\one{h}\triangleright c)\sigma(\two{h}, \one{g})\#_{\sigma}\three{h}\two{g}\otimes\\
    &\sigma^{-1}(S(\four{g}), \five{g})\sigma^{-1}(S(\six{g})S(\five{h}), \six{h})c'(S(\seven{g})\triangleright b')\#_{\sigma}S(\three{g})S(\four{h}),
\end{align*}
where in the 4th, 5th and 6th step we use that $H$ is cocommutative and the image of $\sigma$ belongs to the centre of $B$, and the 5th step also uses Proposition \ref{proposition. unital twist}.
On the other hand,
\begin{align*}
    &\phi(X)\phi(Y)\\
    &=(b\#_{\sigma}\one{h}\otimes b'\sigma^{-1}(S(\three{h}), \four{h})\#_{\sigma} S(\two{h}))(c\#_{\sigma}\one{g}\otimes c'\sigma^{-1}(S(\three{g}), \four{g})\#_{\sigma} S(\two{g}))\\
    &=b(\one{h}\triangleright c)\sigma(\two{h}, \one{g})\#_{\sigma}\three{h}\two{g}\otimes\\
    &c'\sigma^{-1}(S(\six{g}), \seven{g})\big(S(\five{g})\triangleright(b'\sigma^{-1}(S(\six{h}), \seven{h}))\big)\sigma(S(\four{g}), S(\five{h}))\#_{\sigma}S(\three{g})S(\four{h})\\
    &=b(\one{h}\triangleright c)\sigma(\two{h}, \one{g})\#_{\sigma}\three{h}\two{g}\otimes\\
    &c'\sigma^{-1}(S(\five{g}), \six{g})(S(\seven{g})\triangleright b')(S(\eight{g})\triangleright\sigma^{-1}(S(\six{h}), \seven{h}))\sigma(S(\four{g}), S(\five{h}))\#_{\sigma}S(\three{g})S(\four{h})\\
    &=b(\one{h}\triangleright c)\sigma(\two{h}, \one{g})\#_{\sigma}\three{h}\two{g}\otimes c'(S(\five{g})\triangleright b')\sigma^{-1}(S(\six{g}), \seven{g})\sigma(S(\eight{g}), S(\six{h})\seven{h})\\
    &\sigma^{-1}(S(\nine{g})S(\eight{h}), \nine{h})\sigma^{-1}(S(\ten{g}), S(\ten{h}))\sigma(S(\four{g}), S(\five{h}))\#_{\sigma}S(\three{g})S(\four{h})\\
    &=b(\one{h}\triangleright c)\sigma(\two{h}, \one{g})\#_{\sigma}\three{h}\two{g}\otimes\\
    &c'(S(\four{g})\triangleright b')\sigma^{-1}(S(\five{g}), \six{g})\sigma^{-1}(S(\seven{g})S(\five{h}), \six{h})\#_{\sigma}S(\three{g})S(\four{h}),
\end{align*}
where in the 3rd, 4th and 5th steps we use that $H$ is cocommutative and the image of $\sigma$ belongs to the centre of $B$, and the 4th step also uses Proposition \ref{proposition. unital twist}. Since $H$ is cocommutative, we have $\phi(X\cdot_{\tilde{\sigma}}Y)=\phi(X)\phi(Y)$.

\end{proof}

 Similarly, we can also show the following theorem:

\begin{thm}\label{theorem. 2 cocycle twist with trivial action}
For a crossed product $B\#_{\sigma} H$, with the image of $\sigma$ belonging to the centre of $B$, if the action of $H$ on $B$ is trivial (i.e. $h\triangleright b=\epsilon(h)b$, for all $h\in H$ and $b\in B$), then there is an invertible normalised 2-cocycle $\tilde{\sigma}$ on $\C(B\#H, H)$, such that
$\phi: \C(B\#H, H)^{\tilde{\sigma}}\to \C(B\#_{\sigma}H, H)$  is an isomorphism of left $B$-bialgebroids, where $\tilde{\sigma}$ is given by (\ref{equ. 2-cocycle on bialgebroid}) and $\phi$ is given by (\ref{equ. isomorphic map between bialgebroid}).
\end{thm}

\begin{proof}
If $H$ measures $B$ with trivial action, and $B\#_{\sigma}H$ is an unital associative algebra with the image of $\sigma$ belongs to the centre of $B$, then we can also get Lemma \ref{lemma. construction of 2 cocycle}, Lemma \ref{lemma. coring map} and Theorem \ref{theorem. 2-cocycle twist} without asking $H$ being cocommutative. Indeed, in this case $B\#H$ is equal to $B\otimes H$ with factorwise multiplication, so by Lemma \ref{lemma. construction of 2 cocycle}
\begin{align}\label{equ. 2-cocycle on bialgebroid 1}
    \tilde{\sigma}(b\#\one{h}\otimes b'\# S(\two{h}), c\#\one{g}\otimes c'\# S(\two{g})):=bc\sigma(h, g)c'b',
\end{align}
for all $b\#\one{h}\otimes b'\# S(\two{h})$, $c\#\one{g}\otimes c'\# S(\two{g})\in \C(B\# H, H)$. In this case, Lemma \ref{lemma. coring map} can be also shown, since even without the cocommutativity of $H$ we can show that $\phi$ is right $B$-linear. Moreover, in the proof of Theorem \ref{theorem. 2-cocycle twist} we can see on the one hand
\begin{align*}
    &\phi(X\cdot_{\tilde{\sigma}}Y)\\
    &=b(\one{h}\triangleright c)\sigma(\two{h}, \one{g})\#_{\sigma}\three{h}\two{g}\otimes\\
    &(S(\six{g})S(\seven{h}))\triangleright\big(\sigma^{-1}(\eight{h}, \seven{g})((\nine{h}\eight{g})\triangleright c')(\ten{h}\triangleright b')) \big)\sigma^{-1}(S(\four{g})S(\five{h}), \six{h}\five{g}))\#_{\sigma}S(\three{g})S(\four{h})\\
    &=bc\sigma(\one{h}, \one{g})\#_{\sigma}\two{h}\two{g}\otimes\\
    &\sigma^{-1}(\seven{h}, \eight{g})c'b'\sigma^{-1}(S(\five{g}), \six{g})\sigma^{-1}(S(\five{h}), \six{h}\seven{g})\sigma(S(\four{g}), S(\four{h}))\#_{\sigma}S(\three{g})S(\three{h})\\
    &=bc\sigma(\one{h}, \one{g})\#_{\sigma}\two{h}\two{g}\otimes\\
    &c'b'\sigma^{-1}(S(\five{g}), \six{g})\sigma^{-1}(S(\five{h}), \six{h})\sigma(S(\four{g}), S(\four{h}))\#_{\sigma}S(\three{g})S(\three{h}),
\end{align*}
where in the 2nd and 3rd steps we use the fact that action is trivial, and the image of $\sigma$ belongs to the centre of $B$, and the 2nd step also uses
\begin{align*}
    \sigma^{-1}(S(\one{g})S(\one{h}), \two{h}\two{g})=\sigma^{-1}(S(\two{g}), \three{g})\sigma^{-1}(S(\two{h}), \three{h}\four{g})\sigma(S(\one{g}), S(\one{h})),
\end{align*}
which can be derived from Proposition \ref{proposition. unital twist}, the 3rd step also uses Proposition \ref{proposition. unital twist}. On the other hand,
\begin{align*}
    &\phi(X)\phi(Y)\\
    &=b(\one{h}\triangleright c)\sigma(\two{h}, \one{g})\#_{\sigma}\three{h}\two{g}\otimes\\
    &c'\sigma^{-1}(S(\six{g}), \seven{g})\big(S(\five{g})\triangleright(b'\sigma^{-1}(S(\six{h}), \seven{h}))\big)\sigma(S(\four{g}), S(\five{h}))\#_{\sigma}S(\three{g})S(\four{h})\\
    &=bc\sigma(\one{h}, \one{g})\#_{\sigma}\two{h}\two{g}\otimes\\
    &c'b'\sigma^{-1}(S(\five{g}), \six{g})\sigma^{-1}(S(\five{h}), \six{h})\sigma(S(\four{g}), S(\four{h}))\#_{\sigma}S(\three{g})S(\three{h}),
\end{align*}
where in the last step we use the fact that image of $\sigma$ belongs to the centre of $B$ and the action is trivial.
\end{proof}

Recall that a \textit{Galois object} of a Hopf algebra $H$ is a comodule algebra $A$, such that the canonical Galois map is bijective, and $A^{coH}=\mathbb{C}$. As a result of Theorem \ref{theorem. cleft extension} and Theorem \ref{theorem. 2 cocycle twist with trivial action} we have:

\begin{cor}\label{corollary. Schauenburg}
Let $A$ be a cleft Galois object of $H$, then $\C(A, H)$ is isomorphic to $H^{\gamma}$ as Hopf algebra, where $\gamma: H\otimes H\to \mathbb{C}$ is an invertible normalised 2-cocycle on $H$.
\end{cor}
\begin{proof}
By Theorem \ref{theorem. cleft extension}, we know $A\simeq \mathbb{C}\#_{\sigma}H$, then by Theorem \ref{theorem. 2 cocycle twist with trivial action} we know $\C(\mathbb{C}\#_{\sigma}H, H)\simeq \C(\mathbb{C}\#H, H)^{\tilde{\sigma}}$. Since in \cite{HL20} we know $H\simeq \C(H, H)\simeq\C(\mathbb{C}\#H, H)$, where the isomorphic map $f: H\simeq \C(H, H)$
is given by
$f: h\mapsto \one{h}\otimes S(\two{h})$, and its inverse is given by $f^{-1}: g\otimes h\mapsto g\epsilon(h)$. Then there is a 2-cocycle $\gamma$ on $H$ such that $\C(A, H)\simeq H^{\gamma}$.
As we know in \cite{schau} that $\C(A, H)$ is a Hopf algebra, then we get the corollary.
\end{proof}

Corollary \ref{corollary. Schauenburg} is also shown in \cite{schau}, here we can view Theorem \ref{theorem. 2 cocycle twist with trivial action} as a generalisation of it.

Let $\gamma: H\otimes H\to \mathbb{C}$ be an invertible normalised 2-cocycle on a Hopf algebra $H$, and $B=A^{coH}\subseteq A$ be a cleft Hopf Galois extension. It is shown in \cite{ppca} that we can define a $H^{\gamma}$-comodule algebra $A_{\gamma}$ on the same underlying $H$-comodule $A$ (i.e. $\delta^{A}=:\delta^{A_{\gamma}}: A_{\gamma}\to A_{\gamma}\otimes H^{\gamma}$) with a new product given by
\begin{align}
    a\cdot_{\gamma}a':=\zero{a}\zero{a'}\gamma^{-1}(\one{a}, \one{a'}),
\end{align}
for all $a, a'\in A$. Moreover, $B=A_{\gamma}^{coH^{\gamma}}\subseteq A_{\gamma}$ is a Hopf Galois extension, with the translation map given by
\begin{align}\label{equ. deformed translation map}
   \tau_{\gamma}(h):=\tuno{\three{h}}\otimes_{B}\tdue{\three{h}}\gamma(\one{h}, S(\two{h})),
\end{align}
for all $h\in H$. Indeed, since the canonocal map $\chi_{\gamma}$ is a left $A_{\gamma}$-module map, it is sufficient to show
\begin{align*}
    (\chi_{\gamma}\circ \tau_{\gamma})(h)&=\zero{\tuno{\three{h}}}\zero{\tdue{\three{h}}}\otimes \gamma^{-1}(\one{\tuno{\three{h}}}, \one{\one{\tdue{\three{h}}}})\two{\one{\tdue{\three{h}}}}\gamma(\one{h}, S(\two{h}))\\
    &=\tuno{\four{h}}\zero{\tdue{\four{h}}}\otimes \gamma^{-1}(S(\three{h}), \one{\one{\tdue{\four{h}}}})\two{\one{\tdue{\four{h}}}}\gamma(\one{h}, S(\two{h}))\\
    &=1\otimes \gamma^{-1}(S(\three{h}), \four{h})\five{h}\gamma(\one{h}, S(\two{h}))\\
    &=1\otimes h,
\end{align*}
for all $h\in H$, where the 2nd step uses (\ref{equ. translation map 2}), the 3rd step uses (\ref{equ. translation map 3}), and the last step uses the fact that $\gamma$ is a 2-cocycle on the Hopf algebra $H$. We can also see
\begin{align*}
    (\chi_{\gamma}^{-1}\circ \chi_{\gamma})(a'\otimes_{B} a)&=a'\cdot_{\gamma} \zero{a}\cdot_{\gamma}\tuno{\three{a}}\otimes_{B}\tdue{\three{a}}\gamma(\one{a}, S(\two{a}))\\
    &=a'\cdot_{\gamma}(\zero{a}\zero{\tuno{\four{a}}})\otimes_{B}\tdue{\four{a}}\gamma^{-1}(\one{a}, \one{\tuno{\four{a}}})\gamma(\two{a}, S(\three{a}))\\
    &=a'\cdot_{\gamma}(\zero{a}\tuno{\five{a}})\otimes_{B}\tdue{\five{a}}\gamma^{-1}(\one{a}, S(\four{a}))\gamma(\two{a}, S(\three{a}))\\
    &=a'\cdot_{\gamma}(\zero{a}\tuno{\one{a}})\otimes_{B}\tdue{\one{a}}\\
    &=a'\otimes_{B}a,
\end{align*}
for all $a, a'\in A_{\gamma}$, where the 3rd step use (\ref{equ. translation map 2}),  and the last step uses (\ref{equ. translation map 4}). Therefore, $B=A_{\gamma}^{coH^{\gamma}}\subseteq A_{\gamma}$ is a Hopf Galois extension.

\begin{lem}\label{lemma. twist cleft Galois object}
Let $B=A^{coH}\subseteq A$ be a cleft Hopf Galois extension, then $B=A_{\gamma}^{coH^{\gamma}}\subseteq A_{\gamma}$ is also a cleft Hopf Galois extension.
\end{lem}
\begin{proof}
Since $B=A^{coH}\subseteq A$ is a cleft Hopf Galois extension, from Theorem \ref{theorem. cleft extension} there is an isomorphic map $F: A\to B\otimes H$ between left $B$-modules and right $H$-comodules.
Define $F_{\gamma}=F$ on the underlying vector space of $A$. Then we can see $F_{\gamma}$ is left $B$-linear and right $H$-colinear:
\begin{align*}
    F_{\gamma}(b\cdot_{\gamma} a)=F(ba)=bF(a)=b\cdot_{\gamma}F_{\gamma}(a),
\end{align*}
for all $b\in B$ and $a\in A$. Thus $F_{\gamma}$ is left $B$ linear. We can also see
\begin{align*}
    \zero{F_{\gamma}(a)}\otimes \one{F_{\gamma}(a)}=&  \zero{F(a)}\otimes\one{F(a)}\\
    =&F(\zero{a})\otimes\one{a}=F_{\gamma}(\zero{a})\otimes\one{a},
\end{align*}
thus $F_{\gamma}$ is a $H$ comodule map. Therefore, $B=A_{\gamma}^{coH^{\gamma}}\subseteq A_{\gamma}$ is a cleft Hopf Galois extension.
\end{proof}

\begin{cor}
Let $A$ be a cleft Galois object of $H$,  and $\gamma: H\otimes H\to \mathbb{C}$ be an invertible normalised 2-cocycle on $H$, then $\C(A_{\gamma}, H^{\gamma})$ is isomorphic to of $\C(A, H)^{\omega}$, where $\omega$ is an invertible normalised 2-cocycle on $\C(A, H)$.
\end{cor}
\begin{proof}
Since $A$  is a cleft Galois object, from the Lemma \ref{lemma. twist cleft Galois object} we know $A_{\gamma}$ is a cleft Galois object of $H^{\gamma}$.


Therefore, by Corollary \ref{corollary. Schauenburg} there is an invertible normalised 2-cocycle $\rho$ on $H$, such that
\begin{align*}
    \C(A_{\gamma}, H^{\gamma})\simeq(H^{\gamma})^{\rho},
\end{align*}
Since $\C(A, H)\simeq H^{\sigma}$, we have
\begin{align}
    \C(A_{\gamma}, H^{\gamma})\simeq\C(A, H)^{\sigma^{-1}\star \gamma\star \rho},
\end{align}
and $\omega=\sigma^{-1}\star \gamma\star \rho$.
\end{proof}

\vspace{.5cm}
	
\noindent
{\bf Acknowledgment:}

I would like to thank Prof. Dr. Giovanni Landi and Prof. Dr. Shahn Majid for many useful discussion. I am glad to thank Dr. Song Cheng for the proofreading.

\end{document}